\documentclass[12pt,english]{article}
\usepackage[T1]{fontenc}
\usepackage[latin9]{inputenc}
\usepackage{geometry}
\usepackage{refstyle}
\usepackage{amsmath}
\usepackage{amsthm}
\usepackage{amssymb}

\makeatletter

\AtBeginDocument{\providecommand\secref[1]{\ref{sec:#1}}}
\AtBeginDocument{\providecommand\lemref[1]{\ref{lem:#1}}}
\AtBeginDocument{\providecommand\propref[1]{\ref{prop:#1}}}
\AtBeginDocument{\providecommand\subsecref[1]{\ref{subsec:#1}}}
\AtBeginDocument{\providecommand\remref[1]{\ref{rem:#1}}}
\AtBeginDocument{\providecommand\corref[1]{\ref{cor:#1}}}
\AtBeginDocument{\providecommand\thmref[1]{\ref{thm:#1}}}
\RS@ifundefined{subsecref}
  {\newref{subsec}{name = \RSsectxt}}
  {}
\RS@ifundefined{thmref}
  {\def\RSthmtxt{theorem~}\newref{thm}{name = \RSthmtxt}}
  {}
\RS@ifundefined{lemref}
  {\def\RSlemtxt{lemma~}\newref{lem}{name = \RSlemtxt}}
  {}

\theoremstyle{plain}
\newtheorem{thm}{\protect\theoremname}
\theoremstyle{plain}
\newtheorem{lem}[thm]{\protect\lemmaname}
\theoremstyle{plain}
\newtheorem{cor}[thm]{\protect\corollaryname}
\theoremstyle{definition}
\newtheorem{problem}[thm]{\protect\problemname}
\theoremstyle{plain}
\newtheorem*{lem*}{\protect\lemmaname}
\theoremstyle{plain}
\newtheorem{prop}[thm]{\protect\propositionname}
\theoremstyle{plain}
\newtheorem*{fact*}{\protect\factname}
\theoremstyle{plain}
\newtheorem*{prop*}{\protect\propositionname}
\theoremstyle{remark}
\newtheorem{rem}[thm]{\protect\remarkname}
\theoremstyle{definition}
\newtheorem{defn}[thm]{\protect\definitionname}

\usepackage{indentfirst}
\usepackage{algorithm}
\usepackage{hyperref}

\makeatletter
\renewcommand{\ALG@name}{Code Snippet}
\makeatother

\providecommand{\propref}[1]{\ref{prop:#1}}
\renewcommand{\propref}[1]{\hyperref[prop:#1]{Proposition~\ref*{prop:#1}}}

\AtBeginDocument{%
  \renewcommand{\thmref}[1]{\hyperref[thm:#1]{Theorem~\ref*{thm:#1}}}%
  \renewcommand{\corref}[1]{\hyperref[cor:#1]{Corollary~\ref*{cor:#1}}}%
  \renewcommand{\secref}[1]{\hyperref[sec:#1]{Section~\ref*{sec:#1}}}%
  \renewcommand{\lemref}[1]{\hyperref[lem:#1]{Lemma~\ref*{lem:#1}}}%
  \renewcommand{\subsecref}[1]{\hyperref[subsec:#1]{Subsection~\ref*{subsec:#1}}}%
}

\usepackage[dvipsnames]{xcolor}
\usepackage{listings}

\lstdefinelanguage{Kotlin}{
  comment=[l]{//},
  commentstyle={\color{gray}\ttfamily},
  emph={filter, first, firstOrNull, forEach, lazy, map, mapNotNull, println},
  emphstyle={\color{OrangeRed}},
  identifierstyle=\color{black},
  keywords={!in, !is, abstract, actual, annotation, as, as?, break, by, catch, class, companion, const, constructor, continue, crossinline, data, delegate, do, dynamic, else, enum, expect, external, false, field, file, final, finally, for, fun, get, if, import, in, infix, init, inline, inner, interface, internal, is, lateinit, noinline, null, object, open, operator, out, override, package, param, private, property, protected, public, receiveris, reified, return, return@, sealed, set, setparam, super, suspend, tailrec, this, throw, true, try, typealias, typeof, val, var, vararg, when, where, while},
  keywordstyle={\color{NavyBlue}\bfseries},
  morecomment=[s]{/*}{*/},
  morestring=[b]",
  morestring=[s]{"""*}{*"""},
  ndkeywords={@Deprecated, @JvmField, @JvmName, @JvmOverloads, @JvmStatic, @JvmSynthetic, Array, Byte, Double, Float, Int, Integer, Iterable, Long, Runnable, Short, String, Any, Unit, Nothing},
  ndkeywordstyle={\color{BurntOrange}\bfseries},
  sensitive=true,
  stringstyle={\color{ForestGreen}\ttfamily},
}

\makeatother

\usepackage{babel}
\usepackage{listings}
\providecommand{\corollaryname}{Corollary}
\providecommand{\definitionname}{Definition}
\providecommand{\factname}{Fact}
\providecommand{\lemmaname}{Lemma}
\providecommand{\problemname}{Problem}
\providecommand{\propositionname}{Proposition}
\providecommand{\remarkname}{Remark}
\providecommand{\theoremname}{Theorem}

\begin{document}
\title{Finding a solution to the Erd\H{o}s-Ginzburg-Ziv theorem in $O(n\log\log\log n)$
time}
\author{Yui Hin Arvin Leung\thanks{Department of Pure Mathematics and Mathematical Statistics, University
of Cambridge. Email: yhal2@cam.ac.uk}}
\maketitle
\begin{abstract}
The Erd\H{o}s-Ginzburg-Ziv theorem states that for any sequence of
$2n-1$ integers, there exists a subsequence of $n$ elements whose
sum is divisible by $n$. In this article, we provide a simple, practical
$O(n\log\log n)$ algorithm and a theoretical $O(n\log\log\log n)$
algorithm, both of which improve upon the best previously known $O(n\log n)$
approach. This shows that a specific variant of boolean convolution
can be implemented in time faster than the usual $O(n\log n)$ expected
from FFT-based methods.
\end{abstract}

\section{Introduction}

We denote $[n]=\{1,2,\cdots,n\}$. We use $[a,b]$ to denote integers
in this range, inclusive on both ends. Let $A$ and $B$ be two sets.
The sumset $A+B$ is defined as $A+B=\{a+b:a\in A,b\in B\}$. Adding
a single element $x$ to $A$ will be denoted directly as $A\cup\{x\}$,
as we use the shorthand $A+x$ to mean $A+\{x\}=\{x+a|a\in A\}$.

We use $\mathbb{Z}_{n}$ to mean the ring $\mathbb{Z}/n\mathbb{Z}$.
We will use variables such as $a,b,i,j$ to denotes elements in $\mathbb{Z}_{n}$,
and variables such as $c,d,k,k_{i}$ to denote an integer. The product
between an integer and an element of $\mathbb{Z}_{n}$ is the natural
one. 

We denote by $AP(a,k)=\{0,a,2a,3a,\cdots,ka\}$ the arithmetic progression
starting at $0$ with difference $a$. We emphasize that the first
argument is an element $a\in\mathbb{Z}_{n}$, and the second argument
$k$ is an integer, though throughout the algorithim we have $k\in[0,n-1]$.
In this paper, all arithmetic progressions (APs) are assumed to start
at $0$. Note that $AP(a,k)$ contains $k+1$ elements; however, for
simplicity, we call $k$ the length of the AP (and not $k+1$), which
is convenient when discussing the total lengths in a sumset $\sum_{i}AP(a_{i},k_{i})$.
Thus the ``length'' is not equal to the number of elements.

Throughout the paper, it is implicit that $AP(a,k)$ lives in an ambient
group $\mathbb{Z}_{n}$. It is only necessary to discuss when $n$
is prime for our purpose. Thus $\mathbb{Z}_{n}$ is always a field. 

The Erd\H{o}s-Ginzburg-Ziv theorem states that
\begin{thm}
Let $m$ be a positive integer. Let $a_{1},\ldots,a_{2m-1}$ be elements
of the cyclic group $\mathbb{Z}_{m}$. Then there exists a subset
$I\subset\{1,2,\ldots,2m-1\}$, with $|I|=m$, such that $\sum_{i\in I}a_{i}=0$. 
\end{thm}

The common proof of the Erd\H{o}s-Ginzburg-Ziv theorem involves reducing
to a prime $p$ and then proving a similar lemma for the prime $p$.
This is the critical lemma for the proof and algorithm.
\begin{lem}
\label{lem:main-lemma}Let $k\in\mathbb{Z}_{p}$, and let $a_{1},\ldots,a_{p-1}\in\mathbb{Z}_{p}$
be elements such that none of the $a_{i}$ is equal to $0$. Then
there exists a subset $I\subset\{1,2,\ldots,p-1\}$ such that $\sum_{i\in I}a_{i}=k$. 
\end{lem}

Obtaining a solution to this lemma is typically the slowest step,
as Erd\H{o}s-Ginzburg-Ziv can be reduced to this lemma while keeping
the same time complexity. The complete reduction will be provided
in \secref{complete-reduction}.

This is the most important step. Throughout most of the paper (except
\secref{complete-reduction}) we will assume that $n$ is prime (i.e.
$n=p$).

The proof of this lemma follows relatively easily from the famous
Cauchy-Davenport theorem. Although alternative proofs exist, the approach
using Cauchy-Davenport gives the simplest proof. 
\begin{thm}
(Cauchy-Davenport) Let $p$ be a prime and $A,B\subset\mathbb{Z}_{p}$,
then $|A+B|\geq\min(p,|A|+|B|-1)$. 
\end{thm}

\begin{proof}
(to \lemref{main-lemma}). It suffices to note that the sumset $\{0,a_{1}\}+\{0,a_{2}\}+\cdots+\{0,a_{p-1}\}$
contains the entirety of $\mathbb{Z}_{p}$ as its size can be easily
estimated by the Cauchy-Davenport theorem to be $p$. 
\end{proof}
Although the correctness is known, the Cauchy-Davenport theorem does
not directly yield an efficient algorithm. A straightforward method
to find a solution to lemma 2 is to use knapsack dynamic programming,
which runs in $O(n^{2})$ time. More generally, finding the sumset
$\{0,a_{1}\}+\{0,a_{2}\}+\cdots+\{0,a_{k}\}$ is the modular subset
sum problem, and has a randomized $O(n\log n)$ solution (by \cite{randomresult}).
Choi et al. (2022) provided a far simpler $O(n\log n)$ algorithm
for the case of Erd\H{o}s-Ginzburg-Ziv theorem \cite{main}.

Our main contribution in this article is the following: 
\begin{thm}
There is an $O(n\log\log\log n)$ algorithm to find a solution to
\lemref{main-lemma} 
\end{thm}

\begin{cor}
There is an $O(n\log\log\log n)$ algorithm to find a solution to
the Erd\H{o}s-Ginzburg-Ziv theorem. 
\end{cor}

We also present a practical $O(n\log\log n)$ algorithm that is both
conceptually simpler and practically efficient. The $O(n\log\log n)$
algorithm can be considered an important subroutine in the overall
algorithm. While the $O(n\log\log\log n)$ solution is asymptotically
better, for practical values of $n$, the algorithm would have terminated
before reaching the phase that is unique to the $O(n\log\log\log n)$
algorithm. Both algorithms share a considerable number of ideas.

\subsection{Variant of Boolean Convolution}

A key optimisation -- which saves an extra logarithmic factor in
both algorithms -- arises from a novel method for performing a specific
variant of Boolean convolution. While finding $A+B$ generally requires
$O(n\log n)$ time using the Fast Fourier Transform or other methods,
our application does not require $A+B$ to be computed completely.
Instead, we only require that a sufficiently large subset of $A+B$
is known, and this required size is exactly the one guaranteed by
Cauchy-Davenport. We can state our problem as follows: 
\begin{problem}
Given sets $A,B\subset\mathbb{Z}_{p}$ where $|A|+|B|-1\leq p$, determine
a set $S$ such that $S\subset A+B$ and $|S|\geq|A|+|B|-1$. 
\end{problem}

This general statement may admit solutions with different trade-offs.
In our algorithm, we would like to keep $A$ largely as the same set,
so we actually want the set $S$ to be of the form $S=A\cup T$, where
$A\cap T=\emptyset$, with $|T|=|B|$. Additionally, we would like
the algorithm to have a runtime that depends only on $|B|$ and not
on $|A|$, and it would be best if it is linear in $|B|$. In some
variants, it is useful to write $S=(c+A)\cup T$, that is, we allow
a constant shift on the base set $A$ before adding to it.

Since our algorithms are independent of $|A|$, we shall formalize
the ways we may interact with the set $A$. We may, in $O(1)$ time,
perform any of the following three operations: (1) given $x$, query
whether $x\in A$, (2) find any element in $A$, (3) find any element
not in $A$. Note that we do not take advantage of any pre-calculations
of $A$ that may use additive or multiplicative properties.

The idea from \cite{main} can be restated in this framework as follows: 
\begin{lem*}
\label{lem:single-element-packing}If $|B|=2$, then we can find $S=A\cup\{c\}$
in time $O(\log n)$. 
\end{lem*}
Its solution can be briefly described as a modified binary search.
The $O(n\log n)$ algorithm for \lemref{main-lemma} in \cite{main}
then follows by starting from $\{0\}$ and adding each element of
the input set $\{a_{i}\}$ one by one using\lemref{single-element-packing}

Our contribution is the following 
\begin{lem*}
\label{lem:AP-packing}If $B$ is an arithmetic progression, then
we can find $S=A\cup T$ in time $O(\log n+|B|)$. 
\end{lem*}
The solution is substantially modified from its binary search root
and is essentially new.

Equipped with \lemref{AP-packing}, we can now describe how to solve
\lemref{main-lemma}. First, some transformations are performed on
the input set $\{a_{i}\}$, resulting in another input set that has
at most $\frac{n}{\log n}$ distinct values. This step costs $O(n\log\log\log n)$.
This is now equivalent to recovering a sumset consisting of only $\frac{n}{\log n}$
different APs. We can now build from $\{0\}$ using \lemref{AP-packing}
, leading to $O(n)$ time in this step.

It would be interesting to see if the above technique can be extended
to general sets $B$. Unfortunately, we did not find any approach
that performs better than a direct adaptation of the proof of Cauchy-Davenport.
In particular, we could not make it depend linearly on $|B|$. We
leave the adaptation as the following straightforward proposition. 
\begin{prop}
For general $B$, we can find $S=(c+A)\cup T$ in time $O(|B|^{2}+|B|\log n)$. 
\end{prop}

\begin{proof}
Recall that in a standard proof of the Cauchy-Davenport theorem, we
consider the replacement $A\leftarrow A\cup(B+e)$ and $B\leftarrow B\cap(A-e)$.
We can use the same process as in \propref{part-two-usual-binary-search}
to find a good choice of $e$ for which the new $B$ is strictly smaller
and nonempty. This part costs $O(\log n)$ per element. The updates
on $A$ and $B$ can each be done in $|B|$ time directly (by only
iterating over $B$, not $A$). Since this step is needed at most
$|B|$ times, the task can be completed in $O(|B|^{2}+|B|\log n)$. 
\end{proof}
While Boolean convolution problems likely cannot be solved faster
than $n\log n$, for our specific application to the Erd\H{o}s-Ginzburg-Ziv
theorem, it is easier in the following two ways: Let $T=\{0,a_{1}\}+\{0,a_{2}\}+\cdots+\{0,a_{n-1}\}$,
which is the sumset we wish to compute and find solutions for. 
\begin{itemize}
\item The result of the sumset is already known to be the entirety of $\mathbb{Z}_{n}$. 
\item We will consider various operations that involve the summands in $T$.
In fact, a typical proof of the Cauchy-Davenport theorem is to consider
the replacement $A\leftarrow A\cup(B+e)$ and $B\leftarrow B\cap(A-e)$.
This operation sometimes makes the new sumset a proper subset of $A+B$,
which would be undesirable if we intend to compute $A+B$ exactly.
In our case, as long as we can guarantee that the new $T$ still remains
$\mathbb{Z}_{n}$---mostly by keeping the number of elements and
using Cauchy-Davenport---we are able to change the summands with
more freedom. 
\end{itemize}

\section{Overview of Algorithm \label{sec:overview}}

Let $T=\sum_{i\in I}AP(a_{i},c_{i})$ be a general sumset of APs.
We specifically refer to a summation of the form $\sum_{i\in I}AP(a_{i},c_{i})$
as a sumset. By a construction of $k$ for a sumset, we mean an assignment
$\{b_{i}\}$ with $b_{i}\in AP(a_{i},c_{i})$ such that $\sum_{i\in I}b_{i}=k$.
We say that a set inclusion relation of two sumsets 
\[
\sum_{i\in I}AP(a_{i},c_{i})\subset\sum_{i\in I'}AP(a_{i}',c_{i}')
\]
holds constructively if, for every $k\in\mathbb{Z}_{n}$ and any construction
$\{b_{i}\}$ where $b_{i}\in AP(a_{i},c_{i})$, one can efficiently
obtain a corresponding construction for $k$ in $\sum_{i\in I'}AP(a_{i}',c_{i}')$.
Of course, such a relation is only meaningful with construction, as
for the majority of the algorithm this relation resolves to $\mathbb{Z}_{n}\subset\mathbb{Z}_{n}$. 

Using the reduction found in \secref{complete-reduction}, it suffices
to algorithmically establish \lemref{main-lemma}. The algorithm is
divided into two parts. In the first part, the \textit{Transformation}
phase, where we will perform operations described below to modify
our sum $T=\sum_{i\in I}AP(a_{i},c_{i})$. We shall see that some
changes to the $\{a_{i}\}$ lead to a constructive inclusion relation.
In the second part, the \textit{packing} phase, we start from a representation
of a set $\sum_{i=0}^{k}AP(a_{i},c_{i})$ and start to include $AP(a_{i},c_{i})$
for $i=k+1,k+2\cdots$ one by one. 

\subsection{Operations in Transformation Phase}

The transformation phase makes use of the following constructive inclusions. 
\begin{fact*}
(Operation 0) The APs with the same common difference $a$ can be
merged and split freely. That is, constructively, $AP(a,k_{1}+k_{2})=AP(a,k_{1})+AP(a,k_{2})$. 
\end{fact*}
\begin{fact*}
(Operation 1) Let $a,b\in\mathbb{Z}_{p}$, $x,y\in\mathbb{N}$, where
$ay=bx$. Then $AP(a,y)+AP(b,x)\supset AP(a,2y).$
\end{fact*}
\begin{fact*}
(Operation 2) Let $1\leq a,b<n$, and $v,w\in\mathbb{Z}_{n}$. Suppose
$va=wb$, and $v$is coprime to $w$. Let $t\geq v$, $s\geq w$,
then there exist a constant $c\in\mathbb{Z}_{n}$ such that $c+AP(z,u)\subset AP(a,t)+AP(b,s)$,
where $z=\frac{a}{w}=\frac{b}{v}$ taken in $\mathbb{Z}_{n}$. And
$u$ can be taken to be $tw+sv-2(vw-v-w+1)$
\end{fact*}
Details of operation 1 and 2 can be found in \ref{prop:operation-1},
\ref{prop:operation-2} respectively.

In our sum set $\sum AP(a_{i},c_{i})$, let $W=\sum c_{i}$ denote
the total length (note this is sum of (cardinality-1)). Initially,
$W=n-1$. We note that both operations can lead to gains in $W$.
For operation 2, the gain in length is at least $(v+w-2)$ (see note
after \ref{prop:operation-2}). Operation $1$ also gains $W$ if
we operates small-to-large: Operation 1 roughly states that if we
have both $AP(a,y)$ and $AP(b,x)$, we can exchange one for the other.
If we exchanged the $AP(b,x)$ for the $AP(a,y)$ , the gain in $W$
is $y-x$, so we would prefer to have $y>x$. Throughout our algorithm,
we will only exchange the shorter run for the longer run, guaranteeing
gains in $W$.

The reconstruction process after applying the operations may cost
extra time. With our two major operations below, operation 1 takes
$O(1)$ time and operation 2 takes $O(\log x)$ time for reconstruction,
where $x$ is the length of the involved AP. We will only discuss
inclusion relations due to these two operations.

These two operations serve different roles. Briefly, the first operation
is more robust and less destructive, while the second operation has
the potential to gain length in greater efficiency. Using operation-1
directly is enough to establish the first phase of the $O(n\log\log n)$
algorithm, whereas using operation 1 and 2 operation 2 carefully is
needed to establish the $O(n\log\log\log n)$ algorithm.

\subsection{Major Components in Transformation Phase}

Starting from $\sum AP(a_{i},c_{i})$ , our target is to repeatedly
find a valid operation. Recall that $W=\sum c_{i}$ is the total length. 

When $W\geq n$, a valid operation must exist, since the sets $\{a_{i},2a_{i},\cdots,ka_{i}\}$
cannot be all disjoint by considering their size. Our goal is instead
to find these efficiently. If we simply exhaust all possible operations,
this will lead to an $O(n\log n)$ algorithm (see remark after \ref{lem:trim}).
To improve, we must set specific targets, and switch to a different
mode or phases once a target is reached. We will also try to check
for collisions between sets $\{a_{i}\cdots,k_{i}a_{i}\}$ in a certain
order. 

We will keep track of two parameters, $W=\sum k_{i}$, the total length,
and $s$ the ``diversity'', the number of distinct values $i$ such
that $k_{i}>0$. Note that we can (constructively) drop any (preferably
short) $AP(a_{i},k_{i})$ to reduce $s$by $1$, but also reducing
$W$. 

It is desirable to increase $W$ and decrease $s$. We identify three
different subprocess, which performs a series marking and collisions
resolution that changes $T$. Each of the following steps are $O(n)$.
\begin{itemize}
\item Enrichment Step: This step increases $W$ without decreasing $s$.
\item Trim Step: This step reduces $s$ mildly, but also get rid of short
collisions (situations where $ci=dj$ for which $c,d$ are small,
and $A[i]\geq c$, $A[j]\geq d$) , which enables operation 2 to be
applied more effectively. This step doesn't decrease $W$.
\item Growth Step: The main step, this can be performed once $s$is sufficiently
small ($s\leq\frac{n}{1000}$). If $s=\frac{n}{k}$, we obtain the
new $s$to satisfy $s\leq\frac{n}{k^{\frac{4}{3}}}$, and $W$ will
be reset to about $n$.
\end{itemize}
A growth step includes an enrichment step and a trim step followed
by an extra process. The $O(n\log\log\log n)$ algorithm starts by
performing one enrichment and one trim, then perform $\log\log\log n$
steps of growth. This would make $s\leq\frac{n}{\log n}$. We shall
see that this allows phase two to be completed in $O(n)$.

An important simplification is to only use trim step and the pack
phase to give $O(n\log\log n)$ solution. The trim step is considerably
simpler than the enrichment step, hence more practical in real implementations.
The details can be found in \ref{sec:practical}.

\subsection{Second Phase }

The second part can be considered a generalization of the method found
in \cite{main}. The method used in their simple $O(n\log n)$ algorithm
is as follows: 
\begin{prop*}
Let $S\subset\mathbb{Z}_{n}$ and $b\in\mathbb{Z}_{n}$, then we can
find an extra element $c$ such that $S\cup\{c\}\subset S+\{0,b\}$
in $O(\log n)$ time. 
\end{prop*}
It seems that $O(\log n)$ time is the best that can be expected in
this single element case. Our contribution is a generalization that
allows an entire AP to be added at one time, this is a restatement
of \lemref{AP-packing}. 
\begin{prop*}
Let $S\subset\mathbb{Z}_{n}$ , $|S|\leq n-k$ and a given $AP(b,k)$.
Then, we can find a set of values $T$ disjoint from $S$, which has
$|T|=k$, such that $S\cup T\subset S+AP(b,k)$ in $O(\log n+k)$
time. 
\end{prop*}
Since after phase $1$, we only have $\frac{n}{\log n}$ APs, we can
include each of them using this lemma in a total of $O(n)$ time.

The rest of the article is organized as follows: We explain the trim
and enrichment step in section 4, the growth step in section 5, the
packing phase in section 6, the reduction from full Erd\H{o}s-Ginzburg-Ziv
theorem in section 7 and discussing the implementation in section
8.

\section{Transformation Phase: Trim and Enrichment steps \label{sec:two-main}}

The algorithm to solve \lemref{main-lemma}starts with a given sumset
$\sum_{i\in\mathbb{Z}_{n}}S(a_{i},c_{i})$ . In phase $1$, we perform
a series of constructive transformations on the sumset until it attains
the desired form. The two following operations can establish inclusions
constructively. 

Firstly, we should represent the sumset $\sum_{i\in\mathbb{Z}_{n}}S(a_{i},c_{i})$
as an array of size $n$. We set $A[i]=c_{i}$ if $a_{i}$ is present,
or $0$ if not. In other words, $A[i]$ is the multiplicity of $i$
in our sumset. All operations can be described by operating on $A$
directly.

Two additional parameters surronding $A$ is also importnat. Define
$W$ to be sum of $A$, and $s$ to be the number of non-zero elements
in $A$. As mentioned in \secref{overview}, our goal in general is
to increase $W$ and decrease $s$.

\subsection{Core Operations}

We begin by examining the possible fundamental operations. 
\begin{fact*}
(Operation 0) Split or merge $AP(a,k_{1}+k_{2})=AP(a,k_{1})+AP(a,k_{2})$. 
\end{fact*}
Opeartion $0$ is implicitly applied through our representation of
the sumset using $A$.

In order to produce a component of $AP(a,x)$ for use in the two operations
that follow, it suffices to have $A[a]\geq x$, and subtract $x$
from $A[a]$.

\begin{fact*}
(Operation 0- Drop) We can modify $A[i]$ to any new $s$ where $0\leq s\leq A[i]$. 
\end{fact*}
This is by dropping unneeded elements. 

\subsubsection{Operation 1 \protect 
}\begin{prop}
(Operation 1) \label{prop:operation-1}Let $a,b\in\mathbb{Z}_{p}$,
$x,y\in\mathbb{N}$, where $ay=bx$. Then $AP(a,y)+AP(b,x)\supset AP(a,2y)$.
Throughout this paper, we impose an extra requirement of $y>x$.
\end{prop}

\begin{proof}
Consider an arbitrary element $ka\in AP(a,2y)$ (thus $k\in[0,2y]$).
If $k\in[0,y]$, then we can write $ka=(ka+0),$where $ka\in AP(a,y)$
and $0\in AP(b,x)$. If $k\in[y,2y]$, then we can write $ka=(k-y)a+ya,$where
$(k-y)a\in AP(a,y)$ and $ya=xb\in AP(b,x)$. 
\end{proof}
Operation $1$ expressed in terms of $A$ is the following:
\begin{quotation}
({*}) If $ay=bx$, and $A[a]\geq y$ and $A[b]\geq x$, then update
by setting $A[b]\leftarrow A[b]-x$ and $A[a]\leftarrow A[a]+y$. 
\end{quotation}
We denote such a type-1 operation as an $(a,b)$ operation. Each successful
$(a,b)$ operation increases $A[a]$ while decreasing $A[b]$. 

The cost to reconstruct from operation 1 is $O(1)$. The detail can
be found in \subsecref{Recovery-of-Constructions}.

The change in $W$ is $y-x$. Using the extra requirement of $y>x$,
$W$ will be positive. 

We can roughly explain the efficency of operation 1. Assume that $y>x$,
the increase in $W$ is $y-x$, and this can be small if $y$ and
$x$are of similar size. When applying this operation, we frequently
need to undo the operations corresponding to $AP(b,x)$, which costs
$O(x)$. The means overall we need to spend $O(x)$ time for possibly
just $1$ in $W$, making operation 1 not an effective way to gain
$W$. The more important component is that $W$ is non-decreasing
after the operation.

\subsubsection{Operation 2}

We begin by proving the validity of operation $2$.
\begin{prop}
\label{prop:operation-2}Let $1\leq a,b<n$, and $v,w\in[n-1]$. Suppose
$va=wb$, and $v$ is coprime to $w$. Let $t\geq v$, $s\geq w$,
then there exist a constant $c\in\mathbb{Z}_{n}$ such that $c+AP(z,u)\subset AP(a,t)+AP(b,s)$,
where $z=\frac{a}{w}=\frac{b}{v}$ taken in $\mathbb{Z}_{n}$. And
$u$ can be taken to be $tw+sv-2(vw-v-w+1)$
\end{prop}

\begin{proof}
Without loss of generality, we may rescale such that $z=1$ (multiply
everything by $z^{-1}$). Since $z=\frac{a}{w}=\frac{b}{v}$ we have
$v$ is same as $b$ interpreted as an integer, and $w$ is same as
$a$ interperted as integer. The remainder of the proof operate inside
the integers. We claim that everything between $[ab-a-b+1,(ta+sb)-(ab-a-b+1)]$
are inside $AP(a,t)+AP(b,s)$. 

Let $g\in[ab-a-b+1,(ta+sb)-(ab-a-b+1)]$ , using the Frobenius Coin
problem, we find some 

\[
g=a'a+b'b
\]

Where $a',b'\geq0.$

Suppose $a'\leq t$ and $b'\leq s$, we are done. Otherwise, we cannot
have both $a'>t$ and $b'>s$ both, since $g\leq at+bs$. Without
loss of generality, assume $a'>t$, we consider the largest value
$a''$ such that $a''\leq t$ and $a'\equiv a''(\mod b)$. If $a''=a'-kb$,
then by setting $b''=b'+ka$, we see that $a''a+b''b=g$ still, and
$a''\in[t-b+1,t]$, and we have $0\leq a''\leq t,0\leq b''$. It suffices
to check that $b''\leq s$, but this follows as otherwise $g\leq(s+1)b+(t-b+1)a=(ta+sb)-ab+b+a$,
violating $g\leq(ta+sb)-(ab-a-b+1)$

The size of $u$ can be seen to be $(ta+sb)-2(ab-a-b+1)$. We obtain
the size expression by replacing $a,b$ by $w,v$.
\end{proof}
\begin{rem}
The size of $u$ is always at least $2(v+w-1)$, Thus $W$ increases
by at least $v+w-2$. Since we can assume $\min(v,w)\geq1$, $\max(v,w)\geq2$,
we can also write this as $W$ always increase by at least $\frac{v+w}{3}$
for convenience. 
\end{rem}

\begin{prop}
Suppose we also have $t\geq2v,s\geq2w$ in the previous proposition,
then we can take $u\geq2vw$
\end{prop}

\begin{proof}
We need to check $tw+sv-2(vw-v-w+1)\geq2vw$. The left hand-side is
at least $2vw+2wv-2vw+2v+2w-2\geq2vw+2v+2w-2\geq2vw$.
\end{proof}
The previous remarks give the situation for $t,s$ very similar to
$b,a$, giving a linear gain in $W$. The proposition shows quadratic
gains in $W$ when $u,v$ are larger.

By keeping track of the gcd, we adapt the above operation for non-coprime
$a,b$.
\begin{cor}
\label{cor:operation-2-noncoprime}(Operation 2 for non-coprime) Let
$va=wb$ and $g=gcd(v,w)$, then there exist a constant $c\in\mathbb{Z}_{n}$
such that $c+AP(z,\frac{2vw}{g})\subset AP(a,2v)+AP(b,2w)$ where
$z=\frac{a}{w}=\frac{b}{v}$
\end{cor}

\begin{proof}
Let $v=gv'$ and $w=gw'$, so $v'$ and $w'$ are coprime. Thus $AP(a,2v)+AP(b,2w)=g(AP(a,2v')+AP(b,2w')\supset g(c'+AP(z,2v'w'))=gc'+AP(z,2v'w'g)=gc'+AP(z,\frac{2vw}{g})$,
so taking $c=gc'$ gives the desired result.
\end{proof}
To support construction, we can use extended euclidean algorithm.
Therefore, unlike operation $1$, this step takes $O(\log\min(a,b))$
instead of $O(1)$. This does not increase runtime because the actual
number of operations to be performed is less than $n$ (see \remref{op2-collisions}),
but we must carefully track the recovery cost of operation 2.

When applying operation $2$, we will store and keep track of the
shift $c$. It suffices to add to the total shift in each operation
$2$. When recovering an overall solution, if we wish to find a solution
for $c$ in \ref{lem:main-lemma}, we instead look for consturction
that $c-g$ where $g$ is the total shift. We shall ignore the shift
form this point on and this should not cause any difficulties.

Even though operation $2$ is more effective at increasing $W$, relying
solely on increasing $W$ is insufficient on its own. In order to
guarantee a construction through $W$ alone, we need $W\geq n^{2}$,
and even if we only aim to reduce the number of AP to $\frac{n}{c}$,
there is no better bound than requiring $W\geq cn$, and proceeding
like this for $c=\log n$ will give us no better than $O(n\log\log n)$.
Instead, increasing $W$ will become secondary, and we focus on lowering
the number of APs. Increasing $W$ to an amount linear in $n$ enables
us to discard some APs, and this is the main way why increasing $W$
is helpful.

\subsection{Overview of the Algorithm variants}

Define the weight $W$ as the sum of the array $A$. Let $s$ be the
number of distinct values $a$ for which $A[a]>0$. It is desirable
for $W$ to increase, and for $s$ to decrease. Our overall target
in transformation phase is to reach a state where: 
\[
(\#)W\geq2n,\text{and }s\leq\frac{n}{\log n}
\]

We now give a detailed description of the three types of processes
needed. 
\begin{enumerate}
\item Enrichment Step: Aims to increase $W$. This runs a BFS-like process
and collisions are resolved with operation 2.
\item Trim step: Aims to decrease $s$ and eliminating short collisions.
This runs a BFS-like process and collisions are resolved with operation
1.
\item Growth step: Aims to decrease $s$ quickly. This prepares lots of
short APs and operate on them in 1 pass. A growth step always begin
with both a enrichment and trim step. 
\end{enumerate}

\subsection{The enrichment and trim steps}

The algorithm performs a sequence of operations on $(c,i)$ which
is known as a \textit{marking operation}. This involves marking the
cell $ci$. If a cell is marked twice by $(c,i)$ and $(d,j)$, we
call this a \textit{collision}, and we use the operations to make
sure at least one of the two markings are no longer valid. 

Both algorithm are identical, except from changes in 
\begin{itemize}
\item Termination condition 
\item The choice of operation to resolve collisions 
\item Whether it is necessary to choose the minimum value $c$ in marking
operations $(c,i)$
\end{itemize}
\begin{defn}
(Trim Variant). We call this a trim step when the algorithm is run
with operation 1. The algorithm has a target phase $k$. The algorithm
terminates when the next target value is $(c,i)$ and $c>k$. This
variant must pick the smallest possible $c$ in each marking of $(c,i)$.
\end{defn}

\begin{defn}
(Enrichment Variant). We call this enrichment step when the algorithm
is run with operation 2. The algorithm has a target weight $W$. The
algorithm terminates when the current $W$ is at least the target.
This variant may pick any $c$ in each marking of $(c,i)$
\end{defn}

\subsubsection{Initial Conditions}

We begin the description of the algorithm. We initialize the data:
\begin{itemize}
\item $A[i]$ , integer array, the array encoding, taken as input. 
\item $C[i]$, integer array, number of multiples of $i$ that are marked,
initially all $0$.
\item $mark[i]$, array of pairs of integers, an array that marks certain
multiples as visited, all entries are initially empty.
\item A constant shift $g$, where initially $g=0$. $g\in\mathbb{Z}_{n}$
\end{itemize}
Each of these arrays are indexed with $i\in\mathbb{Z}_{n}$

We also track two parameters, 
\begin{itemize}
\item $W$, which is always sum of $A$, the total number of elements with
multiplicity. 
\item $s$, which is equal to the number of non-zero elements in $A$
\end{itemize}

\subsubsection{Invariants Maintained\label{subsec:Invariants-Maintained}}

Throughout the algorithm, and in-between each step, the following
invariants must be satisfied. We put these in front since it is needed
to make sense of the algorithm.
\begin{itemize}
\item $mark[i]$ is either empty, or a pair $(k,a)$ where $k\leq A[i]$,
$k\in[1,n-1]$, $a\in\mathbb{Z}/n\mathbb{Z}$ and $i=ka$. 
\item For each of $c=1,2,\cdots C[i]$, the pair $(c,i)$ is marked (so
$A[ci]=(c,i)$ )
\item $C[i]\leq A[i]$
\end{itemize}
The first two are by design. The part $C[i]\leq A[i]$ should also
be clear from the design of the algorithm.

\subsubsection{Main process}

We defer the proof of claim 1 and 2 below to the next subsection.

The algorithm is to repeatedly find any (Enrichment variant), or smallest
(Trim variant) $k$ and a valid $i$(Claim 1: This can be done linearly)
, such that $C[i]<k\leq A[i]$, and then attempt to mark $(c,i)$
on $ci$, we call this one marking operation. 
\begin{enumerate}
\item If $ci$ is not marked, write $mark[ci]=(c,i)$, update $C[i]\leftarrow C[i]+1$. 
\item Suppose we have $mark[ci]=(d,j)$ (Note 1), note that we must have
$gcd(c,d)=1$ (Claim 2). We call this a collision. We now perform
as many operations of the expected types using $i$ and $j$, more
precisely: 
\begin{enumerate}
\item If we are using operation $1$, w.l.o.g. $c>d$. Define $y=\lfloor\frac{A[j]}{d}\rfloor$.
We update $A[i]\leftarrow A[i]+yc$, $A[j]\leftarrow A[j]-yd$. then
we will increase $A[i]$ (where w.l.o.g. $c>d$) according to \propref{operation-1}.
\item If we are using operation $2$, we set $A[z]\leftarrow A[z]+(A[i]\times d+A[j]\times c)-2(cd-c-d+1)$
(where $z$ is defined as in \propref{operation-2}), we also updates
$g$ by the given shift in \propref{operation-2}. Then, we set $A[i]=0,A[j]=0$.
\end{enumerate}
Cleanup: After the changes to $A[i]$, for any values $i$ where $C[i]>A[i]$,
we remove the marks corresponding to $(c,i)$ where $A[i]<c\leq C[i]$,
then, we set $C[i]=A[i]$.
\item The algorithm terminates with either 
\begin{itemize}
\item The target phase is reached (Trim)
\item The target $W$ is reached (Enrichment)
\item Any $A[i]\geq n-1$. We end the entire algorithm with Special Case
I below.
\item All $A[i]=C[i]$, so no further step 2 is possible. We end the entire
overall algorithm with Special Case II below.
\end{itemize}
\end{enumerate}
Note 1: Since we terminate whenever $A[i]\geq n-1$, we have $ci=dj$
where $c,d\leq n-1$. $c$ cannot equal $d$ since if $(d,i)$ is
already marked, we are not marking the same pair again. This means
we really do have $i\neq j$.

\subsubsection{Claims of Invariants}

We will collect the claims and prove them in specified order
\begin{enumerate}
\item We can find the smallest required $k$ in $O(1)$ each 
\item In step 2, when attempting to mark $(c,i)$ to $(d,j)$ we have $gcd(c,d)=1$
\item $W$ is non-decreasing throughout
\item For any $(c,i)$ and $(d,j)$ where $c\leq C[i]$, $d\leq C[j]$,
we have $ci\neq dj$
\item Sum of $C$ is at most $n-1$
\item In Trim variant, all marking operation $(c,i)$ have $c\geq d$ for
any existing markings $(d,j)$.
\item In Trim variant, $A[i]$ can only increase during a marking operation
$(c,i)$ for some $c$.
\item In Trim variant, if any $(c,i)$ is unmarked, then $A[i]=C[i]$ for
this $i$. 
\item In Trim variant, if $A[i]=C[i]$ for some $i$, then this statement
will continue to hold for this $i$.
\item The algorithm either continues indefinitely or terminates in one of
the described ways.
\end{enumerate}
\begin{proof}
(1) The two variants perform both findings differently. We assume
correctness of (9).

In enrichment variant, we only need to find any such $i$. We can
maintain a set of candidates. Whenever some $A[i]$ is increased,
$i$ is added to the set. We can keep checking any elements in this
set for marking operations, and once it failed to be valid, we remove
it. Using the number of elements in the set as potential, this method
is clearly $O(1)$ amortized. 

In the trim variant, we can use the set $S_{i}$ to represent values
that have $C[i]<i<A[i]$, starting with $S_{1}=[n]$. As long as $S_{i}$
is nonempty, we can take any value $v\in S_{i}$ and check if $C[i]<v<S_{i}$.
If yes, we can perform the marking operations on it, and add it to
$S_{i+1}$. Otherwise, we delete it from $S_{i}$ (and all subsequent
sets). The deletion is correct due to (9), so $i$ will not belong
to any further sets. The total deletion cost is $O(n)$, so this works
in $O(1)$ amortized also. 

In either cases, the set can be realized by either a HashSet (most
general and well known), or a linked list (simple and performant).

(2) If $gcd(c,d)>1$, let $c=gc',d=gd'$, then with \subsecref{Invariants-Maintained},
$(c',i)$ and $(d',j)$ are both marked as $c'<c,d'<d$, which is
a contradiction as $c'i=d'j$

(3) The only time $A[i]$ changes is after a collision. By the remarks
after \propref{operation-1} or \propref{operation-2}, we see that
$W$ is overall non-decreasing. 

(4) Similar to (2), if $ci=dj$ they would have occupied the same
cell, which is a contradiction

(5) The sum of $C$ is identical to the number of marked cells, since
each cell can only be marked at most once, this is $\leq n-1$

(6)-(9). Suppose (6)-(9) holds in the previous steps. Suppose the
next marking is $(c,i)$, and there exist a mark $(d,j)$ where $c<d$.
Using (9), at no previous moments can we have $A[i]=C[i]$. Consider
the state of marking $(d,j)$, since it is chosen over $(c,i)$, $(c,i)$
must not had been a valid choice: Either $c\leq C[i]$, or $c>A[i]$.
If $c\leq C[i]$, then $(c,i)$ is already marked, using (8) and (9),
$A[i]=C[i]$ just before the current marking, then we cannot have
chosen $(c,i)$ for the current marking. If $c>A[i]$, then consider
$(A[i],i)$. If $A[i]\geq C[i]$, then $A[i]=C[i]$ and as previously
shown, this leads to contradiction and we cannot choose $(c,i)$ for
the current marking, so $(A[i],i)$ had not been marked. Apply the
same argument again on $(A[i],i)$ which is possible as $d>c>A[i]$,
we must now have $A[i]\leq C[i]$ or $A[i]>A[i]$, second of which
is impossible. Thus, we have $A[i]\leq C[i]$, forcing $A[i]=C[i]$,
and leads to contradiction. This proves (6).

(7) The only way $A[i]$ changes is if it is part of the type-1 operation.
Suppose the current marking operation is $(c,i)$ and there exists
a mark $(d,j)$, we have $c\geq d$ by $(6)$, thus by definition
of operation 1, we would increase $A[i]$ and decrease $A[j]$. This
shows that the current target of marking $i$ can only ever increase
$A[i]$ increases and no other entries. This proves (7)

(8) An unmarking can only happen as part of the cleanup in the step
2. During an unmarking from $(c,i)$ and $(d,j)$, $A[j]$ changes
by $A[j]\rightarrow A[j](\mod d)$. As $C[i]\geq d$ before the unmarking,
and $\text{New }A[i]<d$ , we must have $C[i]>\text{New }A[i]$, which
by cleanup of step 2, will have $C[i]$ set to $A[i]$

(9) Because $C[i]=A[i]$, $i$ can no longer be chosen as marking
operations, this means $A[i]$ will no longer increase by $(7)$.
If $C[i]$ decrease, it must be deal to an unmarking, so apply $(8)$
again to see that $A[i]=C[i]$ after this.

(10) The algorithm must start with $W\geq n-1$. If there are no collisions,
and $W=n-1$ still holds, then special case II must apply. In other
cases, we cannot reach a state of $A[i]=C[i]$ without any collisions. 

\end{proof}
\begin{prop}
\label{prop:trim-variant-alternate}The trim variant is also equivalent
to the following: In phase $k$, mark all $(k,i)$ for which $i\leq A[i]$,
then perform the phases from $1$ to $n$ in order. Note that the
operation $1$ may disable some value $i$from satisfying $i\leq A[i]$
even if it holds at the start of the phase. We simply skip over any
values where $i\leq A[i]$ no longer holds as they are processed.
\end{prop}

\begin{proof}
Using Invariant (6), all marking operations can be sorted in increasing
order of $c$. We simply call those who have $c=k$ the phase $k$.
The additional descriptions follow easily.
\end{proof}

\subsubsection{Special Cases \label{subsec:Special-Cases}}

We explain the special cases where the algorithm should be terminated
early. In fact, it must be terminated otherwise it risks behaviors
violating note 1.

\paragraph{Special Case I }

If any $A[i]\geq n-1$, this corresponds to a problem 2 of needing
to represent any value $c$as $0$to $n-1$ copies of $i$. Since
$n$ is a prime, this is always possible.

\paragraph{Special case II}

If $A[i]=C[i]$ for all $i$, since there are no collisions, it follows
that the sets $\{i,2i,\cdots A[i]i\}$ must be disjoint. Since the
algorithim begins with $W\geq n-1$, the total size is also at least
$n-1$. Their union is a subset of $[1,\cdots n-1]$ ($0$ does not
belong as otherwise Special Case I applies), thus the two sets must
be equal. In other words, for any target $c\in\mathbb{Z}/n\mathbb{Z}$,
we can represent it as some multiple $ki$ where we have at least
$k$ copies of $i$. 

\subsubsection{Claims of Complexity}

For the Enrichment variant, we need to special handle the following:
If a collision is such that the new $W$ will exceed the target, we
can perform the changes to $A[i]$ and $g$, and terminate the algorithm
immediately (without tracking the $C[i]$), as we are ultimately only
concerned about $A[i]$. This skips the potentially large extra work. 
\begin{lem}
\label{lem:enrich} Let $r\geq1$ be a given target weight multiple.
The Enrichment variant achieves the following. Starting with given
$W$ and $s$, we reach a state with new $W',s'$ where $W'\geq rn$
and $s'\leq s$ in $O(rn)$ time. 
\end{lem}

\begin{proof}
Consider the potential of $W+E$, where $E$ is the number of marked
cells, $W$ is the length.

We see that the algorithm can proceed normally in Invariant (10) except
when reaching the two special conditions. If a marking operation has
no collision, we increased $E$ by $1$, in $O(1)$ time. If a marking
operation finds a collision between $(c,i)$ and $(d,j)$, then we
will perform operation $2$ (\ref{prop:operation-2}) with $v=A[i]$
and $w=A[j]$. The time cost to complete the operation is $O(1)$
to update $A$ and $c$, and $O(A[i]+A[j])$ to unmark all multiples
of $i,j$. Using the estimation on $W$ after \propref{operation-2},
$W$ always increase by $\frac{v+w}{3}$ , which is $\frac{A[i]+A[j]}{3}$
for our current operation. Since $W$ is bounded above from $rn$,
the total time we need to unmark cells are at most $3rn$. 

Now we examine the recovery cost. The cost is $O(\log(min(A[i],A[j]))$.
Clearly, there is a universal constant $c$ for which $log(min(A[i],A[j]))\leq c\frac{A[i]+A[j]}{3}$,
thus the time cost due to reconstruction is bounded by $(rn)\times c$.

Finally, for the claim on effect of $s$, we note that each operation
$2$ always delete all elements on two APs, and possibly add on one.
The value of $s$ must therefore be non-increasing. 
\end{proof}
For the trim variant, we need the concept of phase, which is \propref{trim-variant-alternate}.
We will only perform the algorithm upto phase $k$, which means only
marking and considering multiples $(c,i)$ and $c\leq k$.
\begin{lem}
\label{lem:trim}Let $k\geq1$ be the given target phase. The Trim
variant achieves the following: Starting with given $W$ and $s$,
we reach a state with new $W',s'$ where $W'\geq W$ and $s'\leq s$.
This can be done in time $O(n+min(n\log k,sk))$. The following additional
properties are satisfied: 
\begin{enumerate}
\item if $(c,i),(d,j)$ is any value with $c\leq min(k,A[i])$ and $d\leq min(k,A[j])$
in the new state, then $ci\neq dj$. 
\item the sum of all $A[i]$ among $i$for which $A[i]<k$ is at most $n$
\end{enumerate}
\end{lem}

\begin{proof}
For the time estimates, since unmarking each cells must involve marking
it before, we only need to keep track of the number of cells marked.
We obtain the two estimates by counting how many cells are marked
in each phase. 

In phase $t$, some multiples $(t,i)$ are operated on. As any such
candidates must have $C[i]=t$, and $A[i]>C[i]$, using invariant
(5) we see that the number of such candidates is at most $\frac{n}{t}$.
Sum over the time cost for each phase, this is at most $n(1+\frac{1}{2}+\cdots+\frac{1}{k})=O(n\log k)$.
This gives the first time estimate. 

The second time estimate follows similarly, since there are only $s$active
$AP$, and no marking operations will be performed on the rest (see
invariant 7 and 9), each phase can be completed in $O(s)$. The time
estimate of $O(sk)$ follows. 

An additional $O(n)$ time might be necessary to allocate the necessary
arrays.

For the first extra property, since phase $k$ is completed, for every
value $i$, we have either $A[i]=C[i]\leq k$, or $A[i]>C[i]=k$.
So if $c,d$ satisfy $c\leq min(k,A[i]),d\leq min(k,A[j])$, $(c,i)$
and $(d,j)$ must be marked cells. This shows that they cannot be
equal.

For the second extra property, if $A[i]<k$ for some $i$, then we
must have $A[i]=C[i]$ at this point, and all $A[i]$ multiples of
$i$ have been marked. The sum of such $A[i]$ is therefore no more
than the number of cells marked, which is at most $n-1$ (invariant
5)
\end{proof}
\begin{rem}
These two steps individually give an algorithm for \lemref{main-lemma}.
For the trim step, by completing all $n$ phases, we solve this in
$O(n\log n)$. For the enrichment step, generally we must target $W\geq n^{2}$,
giving an $O(n^{2})$ algorithm. We shall see later that a slight
modification of trim step alone along with phase 2 gives a very practical
and simple $O(n\log\log n)$ algorithm. 
\end{rem}

\section{Transformation Phase: Growth step}

Since our $O(n\log\log\log n)$ algorithm is primarily of theoretical
interest and less practical than the $O(n\log\log n)$ variant, we
do not attempt to refine the constant factors here.
\begin{thm}
\label{thm:op2-main}Let $k\geq1000$. Starting from a sumset where
some $\frac{n}{k}$ values sum to at least $n$, in $O(n)$ time,
we can make operations and reach a situation where some $\frac{n}{k'}$
values sum to at least $n$, and $k'\geq k^{\frac{4}{3}}$.
\end{thm}

\begin{proof}
Firstly, we should run an enrichment step to gain $W$. We may assume
$W$ is now $10n$. This takes $O(n)$ time (\lemref{enrich}). Now
we only keep the information of the new $A$ and forget all other
structural information like $C$ and $mark$.

Next, we will run a Trim step, we run this upto phase $k$, this costs
$O(sk)\leq O(\frac{n}{k})=O(n)$ by \lemref{trim}. We obtain the
following extra properties. If $i,j$ are two APs that are of length
at least $k$, $\frac{i}{j}$ cannot be represented as a fraction
with small denominator and numerator. To be precise, we need $\frac{i}{j}\neq\frac{a}{b}$
for any $1\leq a,b\leq k$, this is true after a phase $k$ trim step
by \lemref{trim}.

Special handling is required when many arithmetic progressions are
already long. Let $I=\{i|A[i]\geq k^{\frac{4}{3}}\}$. If the total
length among AP of $I$ is at least $n$, then we can report these
values directly. Clearly in this case, we can select a subset $I'$
of $I$ such that the total weight within $I'$ is still at least
$n$ and also $|I'|\leq\frac{n}{k^{\frac{4}{3}}}$.

Otherwise, we drop any length that exceeds $k^{\frac{4}{3}}$, that
is, we set $A[i]\leftarrow\min(A[i],k^{\frac{4}{3}})$. The current
$W$ is still at least $9n$ after this step. 

After the two steps, we reach a state where $W\geq9n$ , $s\leq\frac{n}{k}$,
with the no-short-collisions property. We also have all nonzero $A[i]$
having $A[i]\leq k^{\frac{4}{3}}$.

Now, we start a completely different process of marking cells, that
is similar to enrichment step, but we move all the collision result
to a seperate container. 

Let $T=\sum AP(a_{i},k_{i})$ be our given APs, we maintain the ``next''
set $T'$, which is initially an empty sum of APs. As we found collisions
of the form $c+AP(z,t)\subset AP(x,a)+AP(y,b)$, we will remove the
corresponding $AP(x,a)+AP(y,b)$ from $T$ , and add $AP(z,t)$ to
$T'$. If $T_{end}$ is the ending state of $T$, we would have found
$T_{end}+T'\subset T$. In other words, the new AP produced in a collision
resolution will not be explored in the same step, in contrast to \lemref{enrich}.
The constant shift $c$can be easily kept tracked of, and the total
shift $c'$ is to be subtracted from the final construction target
$k$. 

Instead of marking all multiples, we only mark multiples $ci$ of
$i$ satisfying $c\geq\frac{k}{2}$ but also $c\leq\frac{A[i]}{2}$.
It is easy to see that there are at least $2n$ such multiples, since
there are at least $\frac{9}{2}n$ multiples satisfying $c\leq\frac{A[i]}{2}$,
and the condition $c\geq\frac{k}{2}$ removes at most $\frac{n}{2}$
multiples. Suppose there is a collision, so $ax=by$, we know that
$\frac{k}{2}\leq a,b\leq k^{\frac{4}{3}}$, and that $a\leq\frac{A[x]}{2},b\leq\frac{A[y]}{2}$.
The last condition enables \corref{operation-2-noncoprime} to be
applied. Let $g=gcd(a,b)$, if $g\geq k^{\frac{1}{3}}$, then we have
$a'x=b'y$ for $1\leq a',b'\leq k$, contradicting property of the
trim step. So we must have $\frac{2ab}{g}\geq\frac{(\frac{k}{2})(\frac{k}{2})}{k^{\frac{1}{3}}}=\frac{k^{\frac{5}{3}}}{4}$.
We can now resolve the collision by decreasing $A[x],A[y]$, and record
the $AP(z,t)$ obtained in a different array that represents $T'$.
We will also delete all marked multiples of $i,j$. In the worst case,
this removes at at most $2k^{\frac{4}{3}}$ marked cells per operation
2.

We will continue this until we have encountered at least $n/\frac{k^{\frac{5}{3}}}{4}$
collisions, performing the same amount of operation 2s and obtaining
the same amount of AP, each of length at least $\frac{k^{\frac{5}{3}}}{4}$.
We see that collisions must keep occurring before this state. Recall
we start with a total of $2n$ cells to mark, but we would have removed
a total of at most $2k^{\frac{4}{3}}\times n/(\frac{k^{\frac{5}{3}}}{4})\sim\frac{n}{k^{\frac{1}{3}}}$
marked cells, which is less than $n$ (using $k\geq1000$), so we
will continue to have collision all the way till this aim is achieved.
Moreover, the total number of explored cells clearly do not exceed
$2n$, and total number of removals also do not exceed this number,
so this runs in $O(n)$.

After running this step, we obtain a new set $T'$ that contains some
numbers of APs, with sum of length at least $n$ and each of length
at least $\frac{k^{\frac{5}{3}}}{4}$. This still holds after collecting
the APs with the same common difference. By dropping or trimming some
APs as necessary, we can find a set of at most $\frac{n}{(\frac{k^{\frac{5}{3}}}{4})}$
APs with sum of length at least $n$. Since $\frac{k^{\frac{5}{3}}}{4}\geq k^{\frac{4}{3}}$
for $k\geq1000$, we have reached our desired state in $O(n)$.
\end{proof}
\begin{rem}
\label{rem:op2-collisions}The total reconstruction cost in each step
is also $O(n)$. We see that we only need to resolve and reconstruct
at most $\frac{4n}{k^{\frac{5}{3}}}$ collisions. Since each reconstruction
costs no more than $O(\log k^{\frac{4}{3}})$, this is not a concern.
\end{rem}

We can now achieve our main goal for this variant of phase 1.
\begin{cor}
\label{cor:algo-sharp}There is an $O(n\log\log\log n)$ algorithm
that makes operations until some $\frac{n}{\log n}$ AP have sum of
length at least $n$
\end{cor}

\begin{proof}
Firstly, we run an enrichment step to get $W\geq2n$ in $O(n)$ \ref{lem:enrich}.
Then, we run a trim step till phase $1000$. This takes time $O(n(1+\frac{1}{2}+\cdots+\frac{1}{1000}))=O(n)$
by \lemref{trim}.

Using property (2) in \ref{lem:trim}, we can drop all values with
$A[i]<1000$, and this reduces $W$ by at most $n$. We still have
$W\geq n$ after, but now every nonzero $A[i]$ satisfying $A[i]\geq1000$.
Drop any $A[i]$ if we still have $W\geq n$ after, we see that we
reach the condition for $(2)$ with \ref{thm:op2-main} with $k=1000$.

Then, we repeatedly apply \thmref{op2-main}. Starting from $k=1000$,
since each step we have $k\rightarrow k^{\frac{4}{3}}$, it is easy
to see that in $\log\log\log n$ steps, we would have $k\geq\log n$.
The desired condition is now met. This takes a time of $O(n\log\log\log n)$.
\end{proof}

\section{Packing Phase}

After either variant of phase 1, we have $O(\frac{n}{\log n})$ APs.
We are now ready to find a (implicit) representation for each value
$v\in\mathbb{Z}_{n}$. 

In \cite{main}, a set $S$ was considered, which denotes the set
of values where a solution with sum $k$ can be found. Given the $n-1$
values $a_{i}$, each $a_{i}$ is considered in order, each time expanding
$S$ by $1$, by finding a value $c\not\in S$ for which $S\cup\{c\}\subset(S)\cup(S+a_{i})$.
The value $c$ can be found in $O(\log n)$ by a modified binary search
as follows. We included the method for comparison.
\begin{prop}
\label{prop:part-two-usual-binary-search}Let $S\subset\mathbb{Z}_{n}$
with $|S|<n$, and $b\in\mathbb{Z}_{n}$, then we can find an extra
element $c$ such that $S\cup c\subset S+\{0,b\}$ in $O(\log n)$
time. 
\end{prop}

\begin{proof}
Pick any element $s\in S$ and $t\not\in S$. Perform the binary search
as 
\begin{itemize}
\item Start with $x=sb^{-1}$ and $y=tb^{-1}$. If $x>y$, replace $y$
by $y+n$.
\item While $y>x+1$, calculate $mid=\lfloor\frac{x+y}{2}\rfloor$ and examine
$z=(mid)b$:
\begin{itemize}
\item If $z\in S$, write $x\leftarrow mid$.
\item If $z\not\in S$, write $y\leftarrow mid$.
\end{itemize}
\end{itemize}
Observe that in each step of the algorithm, it remains that $xb\in S$
and $yb\not\in S$. When the algorithm terminates, we found that $xb\in S$
but $(x+1)b\not\in S$, which gives $(x+1)b$ a desirable value for
$c$.
\end{proof}
We return to the current algorithm. After completing the first part
of the algorithm, we have $\frac{n}{c}$ APs that have total length
exceeding $n$. We optimise the above process, to instead enable adding
a whole AP at a time, in $O(\log n+k)$. The modifications to the
above binary search are substantial, in order to address all concerns. 
\begin{lem}
\label{lem:parttwo-algo}Let $S\subset\mathbb{Z}_{n}$ , $|S|\leq n-k$
and a given $AP(b,k)$. Then, we can find a set of values $T$ disjoint
from $S$, which has $|T|=k$, such that $S\cup T\subset S+AP(b,k)$
in $O(\log n+k)$ time. 
\end{lem}

\begin{proof}
Throughout our algorithm, we will maintain $S$ and $S^{c}$ as set-like
structure, so we can find a member in either set in $O(1)$.

Without loss of generality, we can assume that $b=1$ (multiply everything
below by $b^{-1}$). Note that we will only query $S$ for some $x$,
whether or not $x\in S$, and otherwise do not use additive/multiplicative
properties with $S$, so we may assume $b=1$.

Our goal is now to perform the following operation $k$ times: Choose
some $v$ such that $v\in S$, $v+1\not\in S$, and add $v+1$ to
$S$. Let the $k$ added values be $T$, then it easily follows that
$S\cup T\subset S+AP(1,k)$. We use the following fillgap procedure.This
procedure recursively inserts missing elements between $x$ and $y$
until $k$ elements have been added to $S$. Note that at the beginning
of each call to $\texttt{fillgap(x,y)}$, we have $x\in S$ and $y\not\in S$.

Due to its precise nature, we choose to present the algorithim entirely
in puesdocode (this is a mix of Kotlin/Swift/Python).

\begin{lstlisting}
fun fillgap(x,y){
	if(k values have already been added to S) return 
	if(y == x+1 ){
		add y into S
		return
	}

	let mid = midpoint(x,y)

	if(mid is not in S){
		fillgap(x,mid)
	}
	if(y-1 is not in S){
		fillgap(mid, y-1)
	}

	if(less than k values had been added){
		add y to S
	}
}

while(less than k values had been added to S){
	let x = any value in S 
	let y = any value not in S 
	fillgap(x,y)
}
\end{lstlisting}

Define directed length $\texttt{dist(x,y)}$ to be $y-x$ if $y\geq x$,
$y-x+n$ otherwise. $\texttt{midpoint(x,y)}$ is the mod $n$ variant
of the midpoint $\lfloor\frac{x+y}{2}\rfloor$. The directed interval
$[x,y]$ is split into $[x,mid]$ and $[mid,y]$, where $[x,mid]$
is either of the same size or $1$ more than size of $[mid,y]$. We
can compute this formally as follows, first let $y'=y$ if $y>x$,
and $y'=y+n$ if $y<x$. ($y\neq x$ is undefined and will not happen).
We then calculate $mid'=\lfloor\frac{x+y'}{2}\rfloor$, and lastly
let $mid=mid'$if $mid'<n$, and $mid'-n$ otherwise. 

Determining if $k$ values had already been added can be done in $O(1)$
in many ways and may depend on the realisation of $S$. 

We can check that the function is well defined. We only need to pay
attention to guaranteeing $x\neq y$. The base calls start with $x\neq y$.
If $y\geq x+3$ then clearly $\texttt{fillgap(x,mid)}$ and $\texttt{fillgap(mid,y-1)}$
have different arguments. If $y=x+2$, then after $\texttt{fillgap(x,mid)}$,
we must have $mid=y-1\in S$, so we will not recur into $\texttt{fillgap(mid,y-1)}$.

It is easy to see that the stated assumption holds for all recursive
calls, that unless $k$ values had already been added to $S$, at
the beginning of every (main or recursive) call fillgap(x,y), we must
have $x\in S,y\not\in S$. We can also check that at the end of every
fillgap(x,y), we must have $y\in S$. From this, the correctness that
every time we add $y\in S$ we must have $y-1\not\in S$ can be easily
verified. It is also clear that at most $k$ values will be added
in total.

As for the time complexity, note that in the fillgap function, unless
our objective is already completed, every return from the fillgap
function must add one value to $S$. Additionally, the maximum recursion
depth is at most $\log n$ since every additional depth must half
the length $y-x$. (By length, we mean the distance in cyclic order,
which is $y-x$ or $y-x+n$ whichever is in $[0,n-1]$.) Hence, by
the $\log n+k$-th call to the fillgap function we must have returned
at least $k$ times, showing that the whole procedure runs in $O(\log n+k)$.
\end{proof}
\begin{rem}
The runtime can be $\Omega(\log n+k)$ because fillgap may recurring
to $\log n$ depth before adding any values to $S$.
\end{rem}

\begin{cor}
\ref{lem:main-lemma} is solved in $O(n\log\log\log n)$
\end{cor}

\begin{proof}
Using \corref{algo-sharp}, it takes $O(n\log\log\log n)$ time to
reach a situation of at most $\frac{n}{\log n}$ APs. We reduce some
lengths of the APs, and make the sum of length to be $n-1$. Starting
from $S=\{0\}$, we add on the $\frac{n}{c}$ APs using \lemref{parttwo-algo}.
The runtime cost is $O(\frac{n}{c}\log n+\sum k)=O(\frac{n}{c}\log n+n)$,
which for $c=\log n$ this is $O(n)$. Altogether, the algorithm is
$O(n\log\log\log n)$. Recovering the construction from this stage
does not add to the time complexity, and details can be found in \ref{sec:Implementations}.
\end{proof}

\section{Practical $O(n\log\log n)$ version \label{sec:practical}}

The $O(n\log\log\log n)$ is impractical for many reasons, a huge
starting cost ($k\geq1000$) before the growth step becomes useful,
and the complicated reconstruction needed for enrichment step. We
include the simple $O(n\log\log n)$ option that uses only Trim step
and phase 2. The trim step can be implemented in a far simpler way
using \ref{prop:trim-variant-alternate}. An implementation of this
practical variant in kotlin can be found on GitHub via https://github.com/arvindf232/EGZ\_algorithm/tree/main.
The core algorithm (especially \ref{lem:main-lemma} only) can be
made very short but plenty of validations have been included. 
\begin{thm}
\ref{lem:main-lemma} can be solved in $O(n\log\log n)$. 
\end{thm}

\begin{proof}
The important difference when using trim step only is that we cannot
guarantee increase $W$, so we cannot assume $W\geq2n$ and drop the
short APs as done in \ref{cor:algo-sharp}. We will need structural
information in the trim step. 

Start by running trim step till phase $k=\log n$, this costs $n\log\log n$.
At this point, the number of cells with $A[i]\geq k$ do not exceed
$\frac{n}{k}$. Instead of starting phase 2 from $\{0\}$, we directly
start from $S$ to be the set of all marked cells. We consider the
first $C[i]$ multiples of each $i$ already included. We include
each value $i$ for which $A[i]>C[i]$ , each of them in $O(\log n+(A[i]-C[i]))$
using \ref{lem:parttwo-algo}, this costs $O(n)$ total. 
\end{proof}

\section{Complete Algorithm \label{sec:complete-reduction}}

For completeness, we include the reduction from Erd\H{o}s-Ginzburg-Ziv
theorem to \lemref{main-lemma}, which is common in various proofs
of Erd\H{o}s-Ginzburg-Ziv theorem.
\begin{prop}
\label{prop:red-prime}The Erd\H{o}s-Ginzburg-Ziv theorem holds for
prime $n$. A solution can be constructed using one call to lemma
2 and $O(n)$ additional time.
\end{prop}

\begin{proof}
Given $2n-1$ values $a_{i}$, we sort them in non-decreasing order,
obtaining $a_{1}\leq a_{2}\leq\cdots\leq a_{2n-1}$. If $a_{i}=a_{i+n}$
for some $i\in[1,n-1]$, then $a_{i}=a_{i+1}=\cdots=a_{i+n}$, and
the theorem follows immediately by choosing $n$ copies of this value.
Otherwise, let $b_{i}=a_{i+n}-a_{i}$ for $i\in[n-1]$, then each
$b_{i}$ is nonzero. We can now apply \lemref{main-lemma}, and find
a subset of $\{b_{1},b_{2}\cdots b_{n-1}\}$ that sums to $-(\sum_{i=1}^{n}a_{i})$
in $O(n\log\log\log n)$. Suppose that $\sum_{i\in I}b_{i}=-(\sum_{i=1}^{n}a_{i})$.
Then, it follows that $\sum_{i\in[n]\backslash I}a_{i}+\sum_{i\in I}a_{i+n}=0$,
which gives the desired $n$ elements that sum to a multiple of $n$. 
\end{proof}
\begin{prop}
\label{prop:red-general}The Erd\H{o}s-Ginzburg-Ziv theorem holds
for all $n$. A solution can be constructed using calls to the previous
proposition and $O(n)$ additional time for a total of $O(n\log\log\log n)$
time.
\end{prop}

\begin{proof}
Proceed by induction. $n=1$ is trivial. If $n$ is prime the previous
proposition applies. Otherwise, let $n=pa$ where $p$ is a prime
and $a\neq1$. Given $2n-1$ values in the multiset $A$ and an initially
empty collection $B$, we proceed as follows. As long as there are
at least $2p-1$ values in $A$, we perform the following: obtain
a subset of $p$ values that sum to some $X$ that is a multiple of
$p$. We remove these $p$ values from $A$ and add $\frac{X}{p}$
to $B$, using the prime case with $n=p$. We can perform this at
least $2a-1$ times since after $2a-2$ operations, we have $2n-1-(p)(2a-2)=2p-1$
values remaining. We now invoke induction hypothesis for $n=a$ on
$B$, which implies that some $a$ values of $B$ sum to a multiple
of $a$. This corresponds to sum of exactly $pa$ values, and its
sum is a multiple of $pa$.

For the complexity, we see that the time estimate $T(n)$ follows
$T(n)=(2a-1)T(p)+T(a)+O(n)$. Note that $a\leq\frac{n}{2}$, so the
contribution from the $O(n)$ term in the recurrence is still overall
$O(n)$ (by geometric summation), similarly, since $T(p)$ is $O(p\log\log\log p)$
by \ref{prop:red-prime}, the contribution from this term is also
$O(n\log\log\log n)$ -- a $T(p)$ $k$ layers deal will be bounded
by $2^{-k}p\log\log\log p$, and summing this over a prime decomposition
of $n$ still gives $O(n\log\log\log n)$. We see that $T(n)$ is
$O(n\log\log\log n)$.
\end{proof}
Note that this reduction is able to ``support'' converting linear
algorithms of \lemref{main-lemma} into a linear algorithm of Erd\H{o}s-Ginzburg-Ziv
theorem. Also note that if we decide to use $O(n\log\log n)$ variant
(\subsecref{core-algo} for phase 1), we will similarly end up with
$O(n\log\log n)$ overall.

\section{Implementations \label{sec:Implementations}}

We believe that for practical values, the $O(n\log\log n)$ algorithm
will be substantially faster than the $O(n\log\log\log n)$ one. It
is also simpler to implement. Therefore, we shall only discuss implementations
of the $O(n\log\log n)$ algorithm. An implementation of this algorithm
in Kotlin can be found in (..some GitHub link...)

We discuss a few implementation choices that simplify the code and
help maintain the desired time complexity.

\subsection{Special Termination of trim and enrichment step\label{subsec:initial-n-1}}

If $W=n-1$, the algorithm might terminate without encountering any
collisions -- thus failing to achieve enrichment targets for an enrichment
step. If the algorithm proceeds without any further marking operation,
this must be detected and especially handled as mentioned in \subsecref{Special-Cases}.

Similarly, $A[i]\geq n-1$ must be handled or there can be overflows.
It suffices to check if any $A[i]\geq n-1$ whenever they are modified.

\subsection{Modular Inverse}

All modular inverses modulo some prime can be computed in total $O(n)$
time. The inverses should be precomputed before the main algorithm
to \lemref{main-lemma}. A intuitive way to do this would be through
finding any primitive root and storing the fact that $(a^{k})^{-1}=(a^{p-1-k})$.
This could also be done with the folklore relation: 

\[
inv[i]=(-\lfloor\frac{n}{i}\rfloor\times inv[n\%i])(\mod n)
\]

which is an rearrangement of the equation 

\[
n=i\times\lfloor\frac{n}{i}\rfloor+(n\%i)
\]

\subsection{Maintaining Sets}

In the proof of invariant (1) in \ref{sec:two-main}, we have described
how to maintain the targets necessary for the trim and enrichment
step. 

We will now discuss the maintaining candidates in packing phase. We
see from \ref{lem:parttwo-algo}, that we need to maintain the set
$S$ and $S^{c}$. It is possible to use a general set implementation
to support their maintenance in $O(n)$. For example, a set based
on doubly-linked list can support membership queries, insertions,
deletions, retrieval of an arbitrary element in $O(1)$ time, and
iterating over it in $O(size)$. This would be sufficient for our
purposes. We could use simpler structures by analyzing that not all
four operations are used, and maintain them differently. The add,
removal and query for containment of $S$ and $S^{c}$ can simply
be handled by a boolean array of size $n$. We also need additional
operations of finding any element of $S$ and $S^{c}$. For $S$,
in our particular variant of the algorithm, $0$ is always an element
of $S$ so we may just report $0$. For $S^{c}$, since elements are
only removed from, we could keep a pointer $pt$ initially having
$pt=0$, and in each find operation, increment $pt$ until $pt\in S^{c}$
and then report $pt$. This works in $O(n)$ amortized. 

\subsection{Sorting}

The sorting required in \propref{red-prime} can be performed in $O(n)$
time using counting sort or radix sort (under our current need to
sort values of at most $n$), rather than the usual $O(n\log n)$
time of comparison-based sorts. 

\subsection{Recovery of Constructions \label{subsec:Recovery-of-Constructions}}

To support the recovery of answers after the procedure of operation
1 in transformation phase, whenever a type-1 operation is performed,
its information is stored. Once a solution is obtained after packing
phase, we can process the operations in reverse, modifying the solution
accordingly, eventually giving a solution satisfying the original
constraints.

To support the recovery of answers in the procedure of packing phase,
we keep track of an additional array $C[i]$, and we set $C[i]=b$
when the value $i$ is added to $A$ during adding of $S(b,k)$. During
recoveries, if we wish to obtain a target $t$, we can repeatedly
add $C[t]$ to our solution and assign $t\leftarrow t-C[t]$.

Due to its theoretical nature , we briefly mention operation $2$.
Recovering answers from the procedure of part I and operation $2$
would require extended Euclidean algorithm to solve for \propref{operation-2},
but we have already seen that this does not add to the time complexity
by \remref{op2-collisions}.

\section{Concluding Remarks}

We discuss a few directions the algorithms can potentially be improved
on.

We have not improved upon the reduction of the theorem to \lemref{main-lemma}.
It may seem that the theorem almost completely reduces to \lemref{main-lemma},
but this is not so. If the values are evenly distributed between $[0,n-1]$,
then we have more freedom about pairing up values in \propref{red-prime}.
Consider the opposite case where most values are equal; for instance,
suppose that in a sequence, $n-1$ values are equal to some $v$.
Let $B$ be the remaining values, and write $C=B-v$. Then, any solution
with exactly $n$ values is equivalent to finding a non-empty subset
of $C$ that sums to $0$. This differs from \lemref{main-lemma}
significantly in that we only want a target of $0$, instead of an
arbitrary target. This can actually be solved in $O(n)$: Consider
the $n+1$ different prefix sums and take the two prefix sums that
have the same value. Therefore, it is conceivable that a better reduction
is possible.

In the time analysis for \lemref{parttwo-algo}, the packing phase
are bounded by $O(n)$ and the transformation phase is bounded by
$O(n\log\log\log n)$. As far as the author is aware, adjusting the
constant choices to balance the two sides may lead to constant factor
improvement, but does not seem to provide anything asymptotically
better than $O(n\log\log\log n)$ overall. 

Our solution did not make use of any involved facts about the cyclic
group $\mathbb{Z}_{p}$. It may happen that the above algorithm actually
runs faster than $O(n\log\log\log n)$ in the worst case if, say,
intersections of APs are carefully considered.

We conjecture that the problem can indeed be solved in $O(n)$; however,
we believe that any approach that requires unmarking already visited
cells is unlikely to achieve this ideal complexity.

\end{document}